\documentclass[12pt]{article}
\usepackage{amsmath,amssymb,amsthm}
\usepackage{hyperref}
\usepackage{tikz}
\usepackage{ifthen}

\def\picEEF{
\begin{tikzpicture}[scale=0.55]
\fill[color=blue!20!white] (3,5) circle (0.4);
\fill[color=blue!20!white] (5,5) circle (0.4);
\fill[color=blue!20!white] (2,4) circle (0.4);
\fill[color=blue!20!white] (4,4) circle (0.4);
\fill[color=blue!20!white] (2,3) circle (0.4);
\fill[color=blue!20!white] (4,3) circle (0.4);
\fill[color=blue!20!white] (1-0.2,2+0.2) circle (0.25);
\fill[color=blue!20!white] (1+0.2,2-0.2) circle (0.25);
\fill[color=blue!20!white] (3,1) circle (0.4);
\fill[color=blue!20!white] (5,1) circle (0.4);
\foreach \i in {0,...,5}{
  \draw[thick] (0.5,\i+0.5)--(5.5,\i+0.5) (\i+0.5,0.5)--(\i+0.5,5.5);
}
\foreach \i in {1,...,5}{
  \foreach \j in {1,...,5}{
    \ifthenelse{\i>\j}{\pgfmathtruncatemacro{\k}{5+\j-\i}}{\pgfmathtruncatemacro{\k}{\j-\i}};
    \node at (6-\i,\j) {\k};
  }
}
\end{tikzpicture}}

\def\picEEE{
\begin{tikzpicture}[scale=0.55]
\fill[color=blue!20!white] (1,4) circle (0.4);
\fill[color=blue!20!white] (2,2) circle (0.4);
\fill[color=blue!20!white] (3,5) circle (0.4);
\fill[color=blue!20!white] (5,3) circle (0.4);
\fill[color=blue!20!white] (4,1) circle (0.4);
\foreach \i in {0,...,5}{
  \draw[thick] (0.5,\i+0.5)--(5.5,\i+0.5) (\i+0.5,0.5)--(\i+0.5,5.5);
}
\foreach \i in {1,...,5}{
  \foreach \j in {1,...,5}{
    \ifthenelse{\i>\j}{\pgfmathtruncatemacro{\k}{5+\j-\i}}{\pgfmathtruncatemacro{\k}{\j-\i}};
    \node at (6-\i,\j) {\k};
  }
}
\end{tikzpicture}}

\def\picE{
\begin{tikzpicture}[scale=.4]
\draw[-] (0,10) to (10,0);
\draw[-] (1,5.2) to (1,9);
\draw[-] (1.2,5) to (5,5);
\draw[-] (6,1.2) to (6,4);
\draw[-] (6.2,1) to (9,1);
\begin{scope}
    \clip (0,0) rectangle (10,10);
    \draw[line width=3pt,color=ISUYellow] (0,10) to (1,9);
    \draw[line width=3pt,color=ISURed] (5,5) to (1,9);
    \draw[line width=3pt,color=ISUYellow] (5,5) to (6,4);
    \draw[line width=3pt,color=ISURed] (9,1) to (6,4);
    \draw[line width=3pt,color=ISUYellow] (9,1) to (10,0);
\end{scope}
\draw[line width=.5pt] (0,0) to (0,10);
\draw[line width=.5pt] (10,10) to (0,10);
\draw[line width=.5pt] (10,10) to (10,0);
\draw[line width=.5pt] (10,0) to (0,0);
\node[] () at  (1,5) {$\times$};
\node[] () at  (6,1) {$\times$};
\fill[color=green] (1,9) circle (0.2);
\fill[color=green](6,4) circle (0.2);
\fill[color=blue](5,5) circle (0.2);
\fill[color=blue](9,1) circle (0.2);
\end{tikzpicture}
}

\def\picG{
\begin{tikzpicture}[scale=.4]
\draw[-] (0,10) to (10,0);
\draw[-] (1,5.2) to (1,9);
\draw[-] (1.2,5) to (5,5);
\draw[-] (5,1.2) to (5,5);
\draw[-] (5.2,1) to (9,1);
\begin{scope}
    \clip (0,0) rectangle (10,10);
    \draw[line width=3pt,color=ISUYellow] (0,10) to (1,9);
    \draw[line width=3pt,color=ISURed] (5,5) to (1,9);
    \draw[line width=3pt,color=ISURed] (9,1) to (5,5);
    \draw[line width=3pt,color=ISUYellow] (9,1) to (10,0);
\end{scope}
\draw[line width=.5pt] (0,0) to (0,10);
\draw[line width=.5pt] (10,10) to (0,10);
\draw[line width=.5pt] (10,10) to (10,0);
\draw[line width=.5pt] (10,0) to (0,0);
\node[] () at  (1,5) {$\times$};
\node[] () at  (5,1) {$\times$};
\fill[color=green] (1,9) circle (0.2);
\fill[color=blue](5,5) circle (0.2);
\fill[color=green](5,5) circle (0.1);
\fill[color=blue](9,1) circle (0.2);
\end{tikzpicture}
}

\def\picF{
\begin{tikzpicture}[scale=.4]
\draw[-] (0,10) to (10,0);
\draw[-] (2,2.3) to (2,8);
\draw[-] (2.7,2) to (8,2);
\begin{scope}
    \clip (0,0) rectangle (10,10);
    \draw[line width=3pt,color=ISUYellow] (0,10) to (2,8);
    \draw[line width=3pt,color=violet] (2,8) to (8,2);
    \draw[line width=3pt,color=ISUYellow] (8,2) to (10,0);
\end{scope}
\draw[line width=.5pt] (0,0) to (0,10);
\draw[line width=.5pt] (10,10) to (0,10);
\draw[line width=.5pt] (10,10) to (10,0);
\draw[line width=.5pt] (10,0) to (0,0);
\node[] () at  (2,2) {\tiny$\times \times$};
\fill[color=green] (2,8) circle (0.2);
\draw[](2,8) circle (0.1);
\fill[color=blue](8,2) circle (0.2);
\draw[color=white](8,2) circle (0.1);
\end{tikzpicture}
}

\def\picA{
\begin{tikzpicture}[scale=.4]
\draw[-] (0,10) to (10,0);
\draw[-] (1,1.2) to (1,9);
\draw[-] (1.2,1) to (9,1);
\draw[-] (3,3.2) to (3,7);
\draw[-] (3.2,3) to (7,3);
\begin{scope}
    \clip (0,0) rectangle (10,10);
    \draw[line width=3pt,color=ISUYellow] (0,10) to (1,9);
    \draw[line width=3pt,color=ISURed] (1,9) to (3,7);
    \draw[line width=3pt,color=violet] (3,7) to (7,3);
    \draw[line width=3pt,color=ISURed] (7,3) to (9,1);
    \draw[line width=3pt,color=ISUYellow] (9,1) to (10,0);
    \end{scope}
\draw[line width=.5pt] (0,0) to (0,10);
\draw[line width=.5pt] (10,10) to (0,10);
\draw[line width=.5pt] (10,10) to (10,0);
\draw[line width=.5pt] (10,0) to (0,0);
\node[] () at  (3,3) {$\times$};
\node[] () at  (1,1) {$\times$};
\fill[color=green] (1,9) circle (0.2);
\fill[color=green] (3,7) circle (0.2);
\fill[color=blue](7,3) circle (0.2);
\fill[color=blue](9,1) circle (0.2);
\end{tikzpicture}
}

\def\picD{
\begin{tikzpicture}[scale=.4]
\draw[-] (0,10) to (10,0);
\draw[-] (2,1.2) to (2,8);
\draw[-] (2.2,1) to (4.8,1);
\draw[-] (5,1.2) to (5,5);
\draw[-] (5.2,1) to (9,1);
\begin{scope}
    \clip (0,0) rectangle (10,10);
    \draw[line width=3pt,color=ISUYellow] (0,10) to (2,8);
    \draw[line width=3pt,color=ISURed] (5,5) to (2,8);
    \draw[line width=3pt,color=violet] (9,1) to (5,5);
    \draw[line width=3pt,color=ISUYellow] (9,1) to (10,0);
\end{scope}
\draw[line width=.5pt] (0,0) to (0,10);
\draw[line width=.5pt] (10,10) to (0,10);
\draw[line width=.5pt] (10,10) to (10,0);
\draw[line width=.5pt] (10,0) to (0,0);
\node[] () at  (5,1) {$\times$};
\node[] () at  (2,1) {$\times$};
\fill[color=green] (2,8) circle (0.2);
\fill[color=green](5,5) circle (0.2);
\fill[color=blue](9,1) circle (0.2);
\draw[color=white](9,1) circle (0.1);
\end{tikzpicture}
}

\def\picB{
\begin{tikzpicture}[scale=.4]
\draw[-] (0,10) to (10,0);
\draw[-] (1,4.2) to (1,9);
\draw[-] (1.2,4) to (6,4);
\draw[-] (4,1.2) to (4,6);
\draw[-] (4.2,1) to (9,1);
\begin{scope}
    \clip (0,0) rectangle (10,10);
    \draw[line width=3pt,color=ISUYellow] (0,10) to (1,9);
    \draw[line width=3pt,color=ISURed] (4,6) to (1,9);
    \draw[line width=3pt,color=violet] (4,6) to (6,4);
    \draw[line width=3pt,color=ISURed] (9,1) to (6,4);
    \draw[line width=3pt,color=ISUYellow] (9,1) to (10,0);
\end{scope}
\draw[line width=.5pt] (0,0) to (0,10);
\draw[line width=.5pt] (10,10) to (0,10);
\draw[line width=.5pt] (10,10) to (10,0);
\draw[line width=.5pt] (10,0) to (0,0);
\node[] () at  (1,4) {$\times$};
\node[] () at  (4,1) {$\times$};
\fill[color=green] (1,9) circle (0.2);
\fill[color=blue](6,4) circle (0.2);
\fill[color=green](4,6) circle (0.2);
\fill[color=blue](9,1) circle (0.2);
\end{tikzpicture}
}

\def\picC{
\begin{tikzpicture}[scale=.4]
\draw[-] (0,10) to (10,0);
\draw[-] (2,5.2) to (2,8);
\draw[-] (2.2,5) to (5,5);
\draw[-] (2,2.2) to (2,4.8);
\draw[-] (2.2,2) to (8,2);
\begin{scope}
    \clip (0,0) rectangle (10,10);
    \draw[line width=3pt,color=ISUYellow] (0,10) to (2,8);
    \draw[line width=3pt,color=violet] (5,5) to (2,8);
    \draw[line width=3pt,color=ISURed] (8,2) to (5,5);
    \draw[line width=3pt,color=ISUYellow] (8,2) to (10,0);
\end{scope}
\draw[line width=.5pt] (0,0) to (0,10);
\draw[line width=.5pt] (10,10) to (0,10);
\draw[line width=.5pt] (10,10) to (10,0);
\draw[line width=.5pt] (10,0) to (0,0);
\node[] () at  (2,5) {$\times$};
\node[] () at  (2,2) {$\times$};
\fill[color=green] (2,8) circle (0.2);
\draw[](2,8) circle (0.1);
\fill[color=blue](5,5) circle (0.2);
\fill[color=blue](8,2) circle (0.2);
\end{tikzpicture}
}

\def\picCa{
\begin{tikzpicture}[scale=.4]
\draw[-] (0,10) to (10,0);
\draw[-] (2,5.2) to (2,8);
\draw[-] (2.2,5) to (5,5);
\draw[-] (2,2.2) to (2,4.8);
\draw[-] (2.2,2) to (8,2);
\begin{scope}
    \clip (0,0) rectangle (10,10);
    \draw[line width=3pt,color=ISUYellow] (0,10) to (2,8);
    \draw[line width=3pt,color=violet] (5,5) to (2,8);
    \draw[line width=3pt,color=ISURed] (8,2) to (5,5);
    \draw[line width=3pt,color=ISUYellow] (8,2) to (10,0);
\end{scope}
\draw[line width=.5pt] (0,0) to (0,10);
\draw[line width=.5pt] (10,10) to (0,10);
\draw[line width=.5pt] (10,10) to (10,0);
\draw[line width=.5pt] (10,0) to (0,0);
\node[] () at  (2,5) {$\times$};
\node[] () at  (2,2) {$\times$};
\fill[color=green] (2,8) circle (0.2);
\draw[](2,8) circle (0.1);
\fill[color=blue](5,5) circle (0.2);
\fill[color=blue](8,2) circle (0.2);
\node[] () at  (9.25,1.65) {$\tiny a$};
\node[] () at  (6.8,4.15) {$\tiny b$};
\node[] () at  (3.8,7) {$\tiny c$};
\node[] () at  (1.45,9.4) {$\tiny d$};
\end{tikzpicture}
}

\newtheorem{theorem}{Theorem}
\newtheorem{lemma}[theorem]{Lemma}
\newtheorem{proposition}[theorem]{Proposition}
\theoremstyle{definition}
\newtheorem*{claim}{Claim}

\definecolor{ISURed}{RGB}{206, 17, 38}
\definecolor{ISUYellow}{RGB}{242, 191, 73}
\definecolor{ISUBrown}{RGB}{188, 94, 30}
\definecolor{forestgreen}{rgb}{0.13, 0.55, 0.15}
\definecolor{mediumpurple}{rgb}{0.4, 0.3, 0.6}

\allowdisplaybreaks

\begin{document}
\title{A generalization of Eulerian numbers\\ via rook placements}
\author{Esther Banaian\thanks{College of St.\ Benedict, Collegeville, MN 56321, USA {\tt embanaian@csbsju.edu}} \and
Steve Butler\thanks{Iowa State University, Ames, IA 50011, USA {\tt butler@iastate.edu}} \and
Christopher Cox\thanks{Carnegie Mellon University, Pittsburgh, PA 15213, USA {\tt cocox@andrew.cmu.edu}} \and
Jeffrey Davis\thanks{University of South Carolina, Columbia, SC 29208, USA {\tt davisj56@email.sc.edu}} \and
Jacob Landgraf\thanks{Michigan State University, East Lansing, MI 48824, USA {\tt landgr10@msu.edu}} \and
Scarlitte Ponce\thanks{California State University, Monterey Bay, Seaside, CA 93955, USA {\tt scponce@csumb.edu}}}
\date{\empty}
\maketitle

\begin{abstract}
We consider a generalization of Eulerian numbers which count the number of placements of $cn$ ``rooks'' on an $n\times n$ board where there are exactly $c$ rooks in each row and each column, and exactly $k$ rooks below the main diagonal.  The standard Eulerian numbers correspond to the case $c=1$.  We show that for any $c$ the resulting numbers are symmetric and give generating functions of these numbers for small values of $k$.
\end{abstract}

\section{Introduction}\label{sec:intro}
Rook placements on boards have a wonderful and rich history in combinatorics (see, e.g., Butler, Can, Haglund and Remmel \cite{rooknotes}).  Traditionally the rooks are placed in a non-attacking fashion (i.e., at most one rook in each row and column) and the combinatorial aspects come from considering variations on the board shapes.

Instead of varying the board, we could also change the restrictions on how many rooks are allowed in each row and each column.  If we have a square board and the number of rooks in each column and row is fixed, then this corresponds to counting non-negative matrices with fixed row and column sums (c.f.\ {\tt A000681}, {\tt A001500}, {\tt A257493}, etc., in the OEIS~\cite{OEIS}).

In this paper, we will look at this latter case of placing multiple rooks in each row and column more closely. We begin in Section~\ref{sec:mjp} by exploring the connections between these rook placements and juggling patterns. In Section~\ref{sec:Eulerian} we look at Eulerian numbers (which correspond to the number of non-attacking rook placements on an $n\times n$ board with a fixed number of rooks below the main diagonal) and in Section~\ref{sec:prop} generalize to the case in which $c$ rooks are placed in each row and each column. In Section~\ref{sec:GF}, we provide generating functions for special cases of these generalized Eulerian numbers. We end with concluding remarks and open problems in Section~\ref{sec:conc}.

\section{Minimal juggling patterns and rook placements}\label{sec:mjp}
Juggling patterns can be described by a \emph{siteswap} sequence listing the throws that the pattern requires, i.e., $t_1t_2\ldots t_n$ where at time $s\equiv i\pmod n$ we throw the ball so that it will land $t_i$ beats in the future.  A sequence of throws can be juggled if and only if there are no collisions, i.e., two balls landing at the same time, which is equivalent to $1+t_1, 2+t_2, \ldots, n+t_n$ being distinct modulo $n$.  One well known property of siteswap sequences is that the average of the throws is the number of balls needed to juggle the pattern (see \cite{DropsDescents,Polster}).

A \emph{minimal juggling pattern} is a valid juggling pattern $t_1 t_2\ldots t_n$ with $0 \leq t_i \leq n-1$.  These form the basic building blocks of juggling patterns since all juggling patterns of period $n$ arise by starting from some minimal juggling pattern and adding multiples of $n$ to the various throws (such additions do not affect modular conditions).  More about this approach is found in Buhler, Eisenbud, Graham and Wright \cite{DropsDescents}.

This naturally leads to the problem of enumerating minimal juggling patterns.  This is done by relating such patterns to rook placements on a square board.  In particular we will consider the $n\times n$  board $\mathcal{B}_n$, with labels on each cell $(i,j)$ given by the following rule:
\[
\left\{\begin{array}{r@{\quad}l}
j-i & \text{if }j\ge i,\\
n+j-i& \text{if }j<i.
\end{array}\right.
\]
We can interpret the rows of $\mathcal{B}_n$ as the \emph{throwing} times (modulo $n$) and the columns of $\mathcal{B}_n$ as the \emph{landing} times (modulo $n$).  The label of the cell $(i,j)$ is then the smallest possible throw required to throw at time $i$ and land at time $j$.

Given a minimal juggling pattern $t_1t_2\ldots t_n$ we form a rook placement by placing a rook in row $i$ on the cell labeled $t_i$ for $1\le i\le n$ (note that this forces the rook to be placed in the column corresponding to the landing time modulo $n$).  Since landing times are unique modulo $n$ no two rooks will be in the same column, so this forms a non-attacking rook placement with $n$ rooks.  Conversely, given a non-attacking rook placement with $n$ rooks we can form a minimal juggling pattern by reading off the cell labels of the covered square starting at the first row and reading down.  This establishes the bijective relationship between minimal juggling patterns and non-attacking rook patterns on $\mathcal{B}_n$.  An example of this is shown in Figure~\ref{fig:bij1} for the minimal juggling pattern $24234$.

\begin{figure}[hftb]
\centering
\picEEE
\caption{A non-attacking rook placement on $\mathcal{B}_5$ corresponding to the minimal juggling pattern $24234$.}
\label{fig:bij1}
\end{figure}

We can extract information about the minimal juggling pattern by properties of the rook placements, including, for example, the number of balls.

\begin{proposition}\label{prop:rooksbelow1}
The number of rooks below the main diagonal in a non-attacking rook placement on $\mathcal{B}_n$ is the same as the number of balls necessary to juggle the corresponding minimal juggling pattern.
\end{proposition}
\begin{proof}
Suppose there are $k$ rooks below the main diagonal in a placement of $n$ non-attacking rooks on $\mathcal{B}_n$. Then when we sum the labels of all the cells covered by a rook, i.e.\ we sum the throw heights for the juggling sequence, we have
\[
\sum_{\ell=1}^n t_\ell = kn + \sum_{j=1}^n j - \sum_{i=1}^n i = kn.
\]
Since the average of the throws is the number of balls needed for the sequence, the claim follows. 
\end{proof}

Note that in Figure~\ref{fig:bij1} there are three rooks below the main diagonal and that the juggling pattern $24234$ requires three balls to juggle.

\subsection{Multiplex Juggling and $c$-rook placements}
A natural variation in juggling is to allow multiple balls to be caught and thrown at a time.  This is known as \emph{multiplex juggling}, and we will see that many of the basic ideas generalize well to this setting.

We will let $c$ denote a hand capacity, i.e.\ at each beat we make $c$ throws (allowing some of the throws to be $0$, which happens when the number of actual balls thrown is less than $c$).  Siteswap sequences of period $n$ now correspond to a sequence of $n$ sets, $T_1T_2\ldots T_n$, where each $T_i$ is a (multi-)set of the form $\{t_{i,1}, t_{i,2}, \ldots, t_{i,c}\}$, denoted in shorthand notation as $[t_{i,1}t_{i,2}\ldots t_{i,c}]$.  A multiplex juggling sequence is valid if and only if the \emph{juggling modular condition} is satisfied.  Namely, every $1\leq\ell\leq n$ appears exactly $c$ times in the multiset
\[
\{t_{i,j}+i\pmod n\}_{{1 \leq i \leq n} \atop {1 \leq j \leq c}}.
\]
In other words, no more than $c$ balls land at each time.\footnote{A $0$ throw indicates a ball is not landing.}  As in standard juggling patterns, the number of balls $b$ needed to juggle the pattern relates to an average.  In particular, 
\[
 \frac1n\sum_{i=1}^n \sum_{j=1}^c t_{i,j} = b.
\]
We say a multiplex juggling sequence is a \emph{minimal multiplex juggling sequence} if and only if $0\le t_{i,j}\le n-1$ for all throws $t_{i,j}$.

There is a relationship between period $n$, hand capacity $c$ multiplex juggling sequences and placements of ``rooks'' on $\mathcal{B}_n$.  This is done by generalizing from non-attacking rook placements to \emph{$c$-rook placements}, placements of $cn$ rooks with exactly $c$ rooks in every row and column, where multiple rooks are allowed in cells.

There is a bijection between minimal multiplex juggling patterns of period $n$ with hand capacity $c$ and $c$-rook placements on $\mathcal{B}_n$.  In particular, for each $i$ we place $c$ rooks in the $i$-th row corresponding to $t_{i,1},\ldots,t_{i,c}$.  Conversely, given a $c$-rook placement we can form a minimal multiplex juggling sequence by letting $T_i$ denote the cells covered by the rooks in row $i$ (with appropriate multiplicity).  An example of this is shown in Figure~\ref{fig:bij2} for the minimal multiplex juggling pattern $[24][02][14][22][03]$.

\begin{figure}[htb]
\centering
\picEEF
\caption{A $2$-rook placement on $\mathcal{B}_5$ corresponding to the minimal multiplex juggling pattern $[24][02][14][22][03]$}
\label{fig:bij2}
\end{figure}

By the same argument used for Proposition~\ref{prop:rooksbelow1} we have the following.
\begin{proposition}\label{prop:rooksbelow2}
The number of rooks below the main diagonal in a $c$-rook placement on $\mathcal{B}_n$ is the same as the number of balls necessary to juggle the corresponding minimal multiplex juggling pattern.
\end{proposition}

For example, the multiplex juggling pattern in Figure~\ref{fig:bij2} requires four balls to juggle.

\section{Eulerian numbers}\label{sec:Eulerian}
The Eulerian numbers, denoted $\big\langle\!{n\atop k}\!\big\rangle$, are usually  defined as the number of permutations of $[n]$, $\pi=\pi_1\pi_2\ldots\pi_n$, with $k$ ascents ($\pi_i<\pi_{i+1}$), or equivalently the number of permutations with $k$ descents ($\pi_i>\pi_{i+1}$).  There is a bijection between permutations of $[n]$ with $k$ descents and permutations with $k$ drops ($i>\pi_i$), so that $\big\langle\!{n\atop k}\!\big\rangle$ also counts permutations of $[n]$ with $k$ drops (see~\cite{DropsDescents}).  Given an $n\times n$ board, with rows and columns labeled $1,2,\ldots,n$, we can use our permutation to form a non-attacking rook placement by placing rooks at positions $(i,\pi_i)$.  A drop in the permutation corresponds to a rook below the main diagonal, so we will call any rook below the main diagonal a \emph{drop}.

By Proposition~\ref{prop:rooksbelow1}, the number of drops in a non-attacking rook placement equals the number of balls necessary for the corresponding juggling pattern. Therefore, $\big\langle\!{n\atop k}\!\big\rangle$ also counts the number minimal juggling patterns of period $n$ using $k$ balls. 

\begin{table}[h]
\centering
\begin{tabular}{c|ccccccc}
$\big \langle \!{n\atop k}\! \big\rangle$\vphantom{$\bigg|$}
&$k{=}0$&$k{=}1$&$k{=}2$&$k{=}3$&$k{=}4$&$k{=}5$&$k{=}6$\\ \hline
$n{=}1$&$1$\\
$n{=}2$&$1$&$1$\\
$n{=}3$&$1$&$4$&$1$\\
$n{=}4$&$1$&$11$&$11$&$1$\\
$n{=}5$&$1$&$26$&$66$&$26$&$1$ \\
$n{=}6$&$1$&$57$&$302$&$302$&$57$&$1$ \\
$n{=}7$&$1$&$120$&$1191$&$2416$&$1191$&$120$&$1$\\
\end{tabular}
\caption{The Eulerian numbers for $1\le n\le7$.}
\label{Euler1table}
\end{table}

The Eulerian numbers have many nice properties, some of which can be seen in Table~\ref{Euler1table}.  For example, they are symmetric, i.e.\  $\big\langle\!{n \atop k}\!\big\rangle = \big\langle\!{n \atop n-k-1}\!\big\rangle$. 
This can be shown by noting if we start with a permutation with $k$ ascents and reverse the permutation, we now have $n-1-k$ ascents (i.e., ascents go to descents and vice-versa; and there are $n-1$ consecutive pairs).  We will give a different proof of this symmetry in the next section using rook placements.

Another well known property of the Eulerian numbers is a recurrence relation.
\begin{proposition}\label{prop:recurrence}
The Eulerian numbers satisfy
$\big\langle\!{n \atop k}\!\big\rangle = (n-k)\big\langle\!{n-1\atop k-1}\!\big\rangle + (k+1)\big\langle\!{n-1 \atop k}\!\big\rangle$.
\end{proposition}
This recurrence is again proven using permutations and ascents. Here, we provide an alternate proof using rook placements and drops.

\begin{proof}
Start by considering a non-attacking rook placement on an $(n-1) \times (n-1)$ board with $k-1$ drops. Add an $n$-th row and $n$-th column, and place a rook in position $(n,n)$. The newly added rook is not below the diagonal and so we have not created any new drops.  We can now create one additional drop by taking any rook (other than the one just added) which is \emph{on or above} the main diagonal, say in position $(i,j)$, move that rook to position $(n,j)$ and move the rook in position $(n,n)$ to position $(i,n)$.  This moves the rook in the $j$-th column below the main diagonal creating a new drop.  Since no other rook moves we now have precisely $k$ drops and a non-attacking rook placement.  Note that there are $(n-1)-(k-1)=n-k$ ways we could have chosen which rook to move, so that in total this gives $(n-k)\big\langle\!{n-1\atop k-1}\!\big\rangle$ boards of size $n\times n$ with $k$ drops.

Now, consider a non-attacking rook placement on an $(n-1)\times (n-1)$ board with $k$ drops. Add an $n$-th row and $n$-th column, and place a rook in position $(n,n)$.  As before we switch, but now only switch with a rook which is \emph{below} the main diagonal (i.e., a drop).  This will not change the number of drops, so the result is a non-attacking rook placement on an $n\times n$ board with $k$ drops. There are $k$ rooks we can choose to switch with, or alternatively, we can leave the $n$-th rook in position $(n,n)$; thus, there are $k+1$ ways to build the desired rook placement, so that in total this gives $(k+1)\big\langle\!{n-1\atop k}\!\big\rangle$ boards of size $n\times n$ with $k$ drops.

Finally, we note that each $n\times n$ board with $k$ drops is formed uniquely from one of these operations.  This can be seen by taking such a board and then noting the location of the rook(s) in the last row and in the last column.  Suppose these are in positions $(i,n)$ and $(n,j)$, respectively.  We then move these rooks to positions $(i,j)$ and $(n,n)$.  This can at most decrease the number of drops by one (i.e, moving the rook in the last column does not affect the number of drops).  Now removing the last row and column gives an $(n-1)\times(n-1)$ board having a non-attacking rook placement with either $k$ or $k-1$ drops.
\end{proof}

\section{Generalized Eulerian numbers}\label{sec:prop}
The \emph{generalized Eulerian numbers}, denoted $\big\langle\!{n\atop k}\! \big\rangle_{\!c}$, are the number of $c$-rook placements on the $n \times n$ board with $k$ drops. Just as the Eulerian numbers count the number of minimal juggling patterns of period $n$ with $k$ balls, the generalized Eulerian numbers count the number of minimal multiplex juggling patterns of period $n$ with $k$ balls and hand capacity $c$.  Notice that the generalized Eulerian numbers reduce to the Eulerian numbers when $c=1$.  In Table~\ref{tab:gendata} we give some of the generalized Eulerian numbers for $c=2$ and $3$.

\begin{table}[!ht]
\centering
\begin{tabular}{r|rrrrrrrrrrr}
$\big\langle\!{n\atop k}\!\big\rangle_{\!2}$\vphantom{$\bigg|$}&$k{=}0$&$k{=}1$&$k{=}2$&$k{=}3$&$k{=}4$&$k{=}5$&$k{=}6$&$k{=}7$&$k{=}8$\\ \hline
$n{=}1$&$1$\\
$n{=}2$&$1$&$1$&$1$\\
$n{=}3$&$1$&$4$&$11$&$4$&$1$\\
$n{=}4$&$1$&$11$&$72$&$114$&$72$&$11$&$1$\\
$n{=}5$&$1$&$26$&$367$&$1492$&$2438$&$1492$&$367$&$26$&$1$
\end{tabular}

\bigskip

\begin{tabular}{r|rrrrrrrrrr}
$\big\langle\!{n\atop k}\!\big\rangle_{\!3}$\vphantom{$\bigg|$}&$k{=}0$&$k{=}1$&$k{=}2$&$k{=}3$&$k{=}4$&$k{=}5$&$k{=}6$&$k{=}7$&$k{=}8$&$k{=}9$ \\ \hline
$n{=}1$&$1$ \\
$n{=}2$&$1$&$1$&$1$&$1$\\
$n{=}3$&$1$&$4$&$11$&$23$&$11$&$4$&$1$\\
$n{=}4$&$1$&$11$&$72$&$325$&$595$&$595$&$325$&$72$&$11$&$1$
\end{tabular}
\caption{Small values of the generalized Eulerian numbers for $c=2$ and $3$.}
\label{tab:gendata}
\end{table}

These numbers appear to satisfy a symmetry property similar to Eulerian numbers.  We will give two proofs of this symmetry, one in terms of rook placements and the other using minimal multiplex juggling patterns.

\begin{theorem}\label{thm:symmetry}
Let $n$, $k$ and $c$ be non-negative integers.  Then $ \big\langle\!{n\atop k}\!\big\rangle_{\!c} = \big\langle\!{n\atop c(n-1)-k}\!\big\rangle_{\!c}$.
\end{theorem}
\begin{proof}
We construct a bijection between the rook placements with $k$ rooks below the main diagonal and those with $c(n-1)-k$ rooks below the diagonal. Consider a rook placement with $c$ rooks in every row and column, and $k$ rooks below the diagonal. Now, shift every rook one space to the right cyclically. Let us consider the number of rooks which are \emph{strictly} above the main diagonal. 
\begin{itemize}
\item All $c$ rooks in the last column were shifted to the first column. So, none of these rooks are above the main diagonal. 
\item All of the $k$ rooks that were initially below the main diagonal are now either on or still below the main diagonal.
\item All other rooks will be above the diagonal. 
\end{itemize}
Since there are $cn$ rooks on the board total, there are $cn - c - k = c(n-1)-k$ rooks above the diagonal after this shift. Finally, we switch the rows and columns of the board. This flips the rook placement across the main diagonal. After this transformation, there are now $c(n-1)-k$ rooks \emph{below} the main diagonal. This composition of transformations is invertible by switching rows and columns then shifting every rook left one space. Thus, the transformation gives a bijection, completing the proof.
\end{proof}

Before we can give the second proof, we must first establish some basic properties of (multiplex) juggling sequences.

\begin{lemma}\label{scale}
If $T_1 T_2 \ldots T_n$ satisfies the juggling modular conditions with hand capacity $c$, and $\alpha \in \mathbb{Z}_n$ with $\gcd(\alpha,n)=1$, then $(\alpha{T}_{1\alpha^{-1}})(\alpha{T}_{2\alpha^{-1}})\ldots(\alpha{T}_{n\alpha^{-1}})$, where
\[
\alpha T_i := \{\alpha t_{i,1},\ldots,\alpha t_{i,c}\},
\]
and the subscripts are taken modulo $n$, also satisfies the juggling modular conditions.
\end{lemma}
\begin{proof}
We have
\[
A=\lbrace \alpha t_{i\alpha^{-1},j} + i \rbrace_{{1 \leq i \leq n}\atop{1 \leq j \leq c}}
=\lbrace \alpha\big( t_{i\alpha^{-1},j} + i\alpha^{-1}\big) \rbrace_{{1 \leq i \leq n}\atop{1 \leq j \leq c}}
=\lbrace \alpha\big( t_{i',j} + i'\big) \rbrace_{{1 \leq i' \leq n}\atop{1 \leq j \leq c}},
\]
where we use that $\gcd(\alpha,n)=1$ so that $\alpha$ is invertible modulo $n$ and as $i$ ranges between $1$ and $n$, then so does $i':=i\alpha^{-1}$.  Since $\lbrace t_{i,j} + i \rbrace_{{1 \leq i \leq n}\atop{1 \leq j \leq c}}$ has $c$ occurrences each of $1$ through $n$ then scaling by $\alpha$ and taking terms modulo $n$ we also have that $A$ will have $c$ occurrences each of $1$ through $n$.
\end{proof}

\begin{lemma}\label{shift}
If $T_1 T_2 \ldots T_n$ satisfies the juggling modular conditions of hand capacity $c$, and $\beta \in \mathbb{Z}$, then $(T_1 + \beta)(T_2 + \beta) \ldots (T_n + \beta)$, where
\[
T_i + \beta := [t_{i,1}+\beta, t_{i,2}+\beta, \ldots, t_{i,c}+\beta],
\]
still satisfies the juggling modular conditions. 
\end{lemma}

\begin{proof}
The multiset $A = \lbrace (t_{i,j} + \beta) + i \rbrace_{{1 \leq i \leq n}\atop{1 \leq j \leq c}}$ is found by taking $\lbrace t_{i,j} + i \rbrace_{{1 \leq i \leq n}\atop{1 \leq j \leq c}}$ and shifting each element by $\beta$. Since $T_1T_2\ldots T_n$ satisfy the juggling modular conditions then so also must $A$. 
\end{proof}

\begin{proof}[Juggling proof of Theorem~\ref{thm:symmetry}]
We show there is a bijection between the minimal multiplex juggling sequences using $k$ balls and those using $c(n-1)-k$ balls for a fixed length $n$ and hand capacity $c$.  So let $T_1 T_2 \ldots T_n$ be a valid minimal multiplex juggling sequence with $k$ balls and hand capacity $c$. By Lemma~\ref{scale} and Lemma~\ref{shift}, if we scale each $T_i$ by $-1$ (reversing the indexing) and add $n-1$ then the resulting sets still  satisfy the modular juggling conditions. In particular we have that the following satisfies the modular juggling condition:
\[
(n-1-T_n)(n-1-T_{n-1})\ldots(n-1-T_1). 
\]
We also note the resulting throws all lie between $0$ and $n-1$ so that this is indeed a minimal juggling pattern.

The number of balls in the new juggling sequence is
\[
\frac1n\sum_{i=1}^n\sum_{j=1}^c \big((n-1)-t_{i,j}\big)= \frac1n\bigg(cn(n-1)-\sum_{i=1}^n\sum_{j=1}^ct_{i,j}\bigg)=c(n-1)-k.
\]
Finally, we note that this operation is its own inverse, and thus gives the desired bijection.
\end{proof}

\section{Generalized Eulerian numbers for small $k$}\label{sec:GF}
We now look at determining the values of the generalized Eulerian numbers $\big\langle\!{n\atop k}\!\big\rangle_{\!c}$ for small $k$.  This depends of course on both $n$ and $c$.  However, for a fixed $k$ there are only finitely many $c$ that need to be considered.  This is a consequence of the following lemma.

\begin{lemma}\label{lem:boundc}
For $c\ge k$ we have $\big\langle\!{n\atop k}\!\big\rangle_{\!c}=\big\langle\!{n\atop k}\!\big\rangle_{\!k}$.
\end{lemma}
\begin{proof}
It will suffice to establish the following claim.
\begin{claim}
Every $c$-rook placement with $k$ drops has at least $c-k$ rooks in every entry on the main diagonal.
\end{claim}
We proceed to establish this by using induction on $k+c$.  For $k+c=1$, the only possible case is $k=0$ and $c=1$ for which there is only one placement, namely one rook in each cell on the main diagonal.

Now assume that we have established the claim for all $k,c$ with $k+c<\ell$, and let $k+c=\ell$.  Let $S$ be a $c$-rook placement with $k$ drops. We can interpret the rook placement as an incidence relationship of a regular bipartite graph.  By Hall's Marriage Theorem, we know we can find a perfect matching in this bipartite graph which corresponds to $T$, a $1$-rook placement contained in $S$. Suppose there are $i$ drops in $T$.  Then, $S - T$ is a $(c-1)$-rook placement with $k - i$ drops. Since $(c-1) + (k-i) < c+k = \ell$, by our induction hypothesis, there are at least $c-k+i-1$ rooks on each entry on the main diagonal in $S - T$, and hence also in $S$. If $i \geq 1$, we are done. If $i = 0$, then $T$ is again the unique $1$-rook placement where every rook is on the main diagonal, so $S$ still has at least $c-k$ rooks on each entry on the main diagonal. 
\end{proof}

This can also be established in terms of minimal multiplex juggling patterns.
\begin{proof}[Juggling proof of Lemma~\ref{lem:boundc}]
If there are $k$ balls, then at each step we can throw at most $k$ balls, i.e., each $T_i$ has at least $c-k$ entries of $0$.  It follows that in the corresponding $c$-rook placement each row has at least $c-k$ rooks on the diagonal.
\end{proof}

We will be looking at the generalized Eulerian numbers $\big\langle\!{n\atop k}\!\big\rangle_{\!c}$ for $k=1,2,3$.  By Lemma~\ref{lem:boundc} this reduces down to only six cases to consider, namely, 
$\big\langle\!{n\atop 1}\!\big\rangle_{\!1}$, 
$\big\langle\!{n\atop 2}\!\big\rangle_{\!1}$, 
$\big\langle\!{n\atop 2}\!\big\rangle_{\!2}$, 
$\big\langle\!{n\atop 3}\!\big\rangle_{\!1}$, 
$\big\langle\!{n\atop 3}\!\big\rangle_{\!2}$ and 
$\big\langle\!{n\atop 3}\!\big\rangle_{\!3}$.  Since $\big\langle\!{n\atop k}\!\big\rangle_{\!1}=\big\langle\!{n\atop k}\!\big\rangle$, then $\big\langle\!{n\atop 1}\!\big\rangle_{\!1}$, 
$\big\langle\!{n\atop 2}\!\big\rangle_{\!1}$ and 
$\big\langle\!{n\atop 3}\!\big\rangle_{\!1}$ have been previously determined (see {\tt A000295}, {\tt A000460} and {\tt A000498}, respectively, in the OEIS \cite{OEIS}).  So that leaves 
$\big\langle\!{n\atop 2}\!\big\rangle_{\!2}$, 
$\big\langle\!{n\atop 3}\!\big\rangle_{\!2}$ and 
$\big\langle\!{n\atop 3}\!\big\rangle_{\!3}$ and in Table~\ref{tab:GF} we give the generating function for these three sequences.  In the remainder of this section we will demonstrate the techniques used to determine the generating functions by working through the case for $\big\langle\!{n\atop 2}\!\big\rangle_{\!2}$. 

\begin{table}[!htb]
\centering
\begin{multline*}
\sum_{n\ge0}{\textstyle\big\langle\!{n\atop 2}\!\big\rangle_{\!2}}x^n=
x^2+11x^3+72x^4+367x^5+1630x^6+6680x^7+26082x^8+\cdots\\
=\frac{x^2-x^3-x^4-3x^5+5x^6}{(1-x)^3(1-2x)^2(1-5x+5x^2)}
\end{multline*}
\vspace{-10pt}
\begin{multline*}
\sum_{n\ge0}{\textstyle\big\langle\!{n\atop 3}\!\big\rangle_{\!2}}x^n=
4x^3+114x^4+1492x^5+13992x^6+109538x^7+769632x^8+\cdots\\
=\frac{\displaystyle{4x^3+2x^4-300x^5+1748x^6-4676x^7+7058x^8-6648x^9\hspace{0.75in}\atop \hspace{2.75in} +4397x^{10}-2206x^{11}+625x^{12}}}{(1-x)^4(1-2x)^3(1-5x+5x^2)^2(1-8x+13x^2)}
\end{multline*}
\vspace{-10pt}
\begin{multline*}
\sum_{n\ge0}{\textstyle\big\langle\!{n\atop 3}\!\big\rangle_{\!3}}x^n=
x^2+23x^3+325x^4+3368x^5+28819x^6+218788x^7+\cdots\\
=\frac{\displaystyle{x^2-7x^3+39x^4-336x^5+1844x^6-5545x^7+9697x^8\hspace{0.95in}\atop\hspace{1.25in}-10404x^9+7532x^{10}-4558x^{11}+2435x^{12}-700x^{13}}}{(1-x)^4(1-2x)^3(1-5x+5x^2)^2(1-10x+27x^2-20x^3)}
\end{multline*}
\caption{Generating functions for some of the generalized Eulerian numbers.}
\label{tab:GF}
\end{table}

\subsection{Placing rooks in a generic rook placement}
We break the problem of counting $c$-rook placements into several sub-problems according to the way the rooks below the main diagonal are placed relative to one another (i.e., relative placements instead of absolute placements).  Given some generic placement of the $k$ rooks below the main diagonal we can determine the number of ways to place the remaining rooks on or above the main diagonal.  We then combine the results over all possible generic placements.

We will carefully work through the rook placement shown in Figure~\ref{fig:basecase} which consists of two rooks below the main diagonal and where both rooks are in the same column and different rows.  Here $a$, $b$, $c$ and $d$ are the number of rows between the various transition points (a transition point to passing a rook, or rooks, in a row or a column as we move along the main diagonal).

\begin{figure}[htbf]
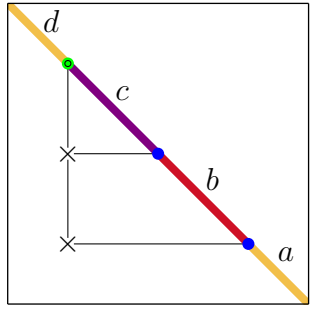

\centering
\picCa
\caption{A $2$-rook placement with two rooks below the main diagonal where both rooks are in the same column and different rows.}
\label{fig:basecase}
\end{figure}

We place the remaining rooks one row at a time starting from the bottom and going to the top.  For each new row, the way we place rooks will depend on all of the choices we have made previously.  However, it suffices to know only what is happening locally.  In particular, we only need to know how many columns can have rooks placed into them, as well as the respective numbers that can go into those columns.  We can represent these by a partition of what we will call the \emph{excess} (the total number of rooks that can still be placed in the columns \emph{after} the row has had its rooks placed).  As we move one row up the board we will gain a new column (from the diagonal) and the excess will change in one of several ways.
\begin{itemize}
\item There are no rooks below or to the left of the new diagonal cell.  Initially we now have a new column that can take up to $c$ rooks, and we place $c$ rooks in the row.  The excess remains unchanged.
\item There are $\tau$ rooks below the new diagonal cell.  Initially we have the new column, but that can only take up to $c-\tau$ rooks (i.e., $\tau$ rooks have already gone into the column), and we still have to place $c$ rooks in the row.  The excess decreases by $\tau$.
\item There are $\sigma$ rooks to the left of the new diagonal cell.  Initially we have the new column that can take up to $c$ rooks, and we place $c-\sigma$ rooks in the row (i.e., $\sigma$ rooks have already gone into the row).  The excess increases by $\sigma$.
\end{itemize}
We note that it is possible for the last two situations to occur simultaneously.

In going from row to row we will transition from partitions of the old excess to partitions of the new excess.  We illustrate this with an example in which case the excesses are both $2$.  We indicate a column which can still have $r$ rooks placed into it by $\boxed{r}$, then underneath look at all possible ways we can place $2$ rooks into those columns, and finally note the resulting set of columns contributing to the new excess.

\bigskip
\noindent\hfil
\begin{minipage}[t]{0.4\textwidth}
\centering
\begin{tabular}{c@{\quad}ccc}
$\boxed2$&$\boxed2$\\[5pt]
$2$&$0$&$\,\to\,$&$\boxed2$\\[5pt]
$0$&$2$&$\to$&$\boxed2$\\[5pt]
$1$&$1$&$\to$&$\boxed1\,\,\boxed1$\\
\phantom{$\boxed1$}
\end{tabular}
\end{minipage}\hfil
\begin{minipage}[t]{0.4\textwidth}
\centering
\begin{tabular}{c@{\quad}c@{\quad}ccc}
$\boxed2$&$\boxed1$&$\boxed1$\\[5pt]
$2$&$0$&$0$&$\,\to\,$&$\boxed1\,\,\boxed1$\\[5pt]
$1$&$1$&$0$&$\to$&$\boxed1\,\,\boxed1$\\[5pt]
$1$&$0$&$1$&$\to$&$\boxed1\,\,\boxed1$\\[5pt]
$0$&$1$&$1$&$\to$&$\boxed2$
\end{tabular}
\end{minipage}
\bigskip

This can be modeled by a transition matrix where the \emph{columns} of the transition matrix correspond to the excess of the original row and the \emph{rows} of the transition matrix correspond to the partitions of the excess of the new row.
\[
\bordermatrix{~ & \boxed2 & \boxed1\,\boxed1\cr~~\boxed2 & 2 & 1 \cr \boxed1\,\boxed1 & 1  &  3\cr}.
\]

Repeating this for all possible situations that might arise for transitioning between excesses $0$, $1$, or $2$, we get the transition matrices in the following table.

\bigskip

\noindent\hfil
\begin{tabular}[b b]{rc}\label{tbl:transitions}
&Transition from \\[7pt]
\rotatebox[origin=c]{90}{Transition to}
&\begin{tabular}{|c||c|c|c|}
\hline
&$\emptyset$&$\boxed1$&$\begin{matrix}\boxed2 & ~\boxed1\,\boxed1\end{matrix}$\\ \hline \hline
$\emptyset$ & $(1)$ & $(1)$ & $\begin{pmatrix}1~~~ & 1\end{pmatrix}$\\ \hline
$\boxed1$ & $(1)$ & $(2)$ & $\begin{pmatrix}2~~~ & 3\end{pmatrix}$\\ \hline
$\begin{matrix}\boxed2 \\\boxed1\,\boxed1\end{matrix}$&
$\begin{pmatrix}1 \\0\end{pmatrix}$& 
$\begin{pmatrix}1 \\1\end{pmatrix}$&
$\begin{pmatrix}2~~~ & 1 \\ 1~~~ & 3 \end{pmatrix}$ \\ \hline
\end{tabular}
\end{tabular}

\bigskip

We now start below the bottom row (in $1$ possible way) and we move up from row to row and multiply on the \emph{left} by the transition that we perform between the two rows.  At any point we stop, the resulting vector will denote the number of ways to fill up the board to that row such with a particular excess.  In particular, if we carry this procedure all the way to the top we will get a $1\times 1$ matrix whose entry is the number of ways to fill in the rooks on and above the main diagonal.

For Figure~\ref{fig:basecase}, where we have of runs of $a$, $b$, $c$ and $d$ rows as well as three other transitions to make, the resulting product that gives our count is as follows
\[
(1)^d\begin{pmatrix}1&1\end{pmatrix}\begin{pmatrix}2&1\\1&3\end{pmatrix}^c\begin{pmatrix}1\\1\end{pmatrix}(2)^b(1)(1)^a.
\]
Finally, for this generic rook placement we sum over all possible choices of $a$, $b$, $c$ and $d$ that gives an $n\times n$ board, i.e.,
\[
\sum_{a+b+c+d=n-3} (1)^d\begin{pmatrix}1&1\end{pmatrix}\begin{pmatrix}2&1\\1&3\end{pmatrix}^c\begin{pmatrix}1\\1\end{pmatrix}(2)^b(1)(1)^a.
\]
In order to help evaluate this sum, we will add in an extra parameter $x$ that keeps track of how many of each transition we made, or viewed another way the power of $x$ corresponds to the number of rows we have.  Therefore when counting the number of placements on an $n\times n$ board, we are interested in the coefficient of $x^n$ of the expression
\[
\sum_{a,b,c,d\ge0} (x)^dx\begin{pmatrix}1&1\end{pmatrix}\begin{pmatrix}2x&x\\x&3x\end{pmatrix}^cx\begin{pmatrix}1\\1\end{pmatrix}(2x)^bx(x)^a.
\]
This sum can be decomposed as a combination of geometric sums giving
\begin{align*}
&\sum_{a,b,c,d\ge0} (x)^dx\begin{pmatrix}1&1\end{pmatrix}\begin{pmatrix}2x&x\\x&3x\end{pmatrix}^cx\begin{pmatrix}1\\1\end{pmatrix}(2x)^bx(x)^a \\
&= x^3\bigg(\sum_{d\ge0}x^d\bigg)\begin{pmatrix}1&1\end{pmatrix}\bigg(\sum_{c\ge0}\begin{pmatrix}2x&x\\x&3x\end{pmatrix}^c\bigg)\begin{pmatrix}1\\1\end{pmatrix}\bigg(\sum_{b\ge0}(2x)^b\bigg)\bigg(\sum_{a\ge0}x^a\bigg)\\
&= x^3\cdot\frac1{1-x}\cdot\begin{pmatrix}1&1\end{pmatrix}\bigg(I-\begin{pmatrix}2x&x\\x&3x\end{pmatrix}\bigg)^{-1}\begin{pmatrix}1\\1\end{pmatrix}\cdot\frac1{1-2x}\cdot\frac1{1-x}\\
&= \frac{x^3}{(1-x)^2(1-2x)}\cdot\begin{pmatrix}1&1\end{pmatrix}\bigg(\frac1{1-5x+5x^2}\begin{pmatrix}1-3x&x\\x&1-2x\end{pmatrix}\bigg)\begin{pmatrix}1\\1\end{pmatrix}\\
&=\frac{x^3(2-3x)}{(1-x)^2(1-2x)(1-5x+5x^2)}.
\end{align*}

This is the generating function for one of the generic ways to place rooks.  We can now repeat this procedure for every way in which we can place rooks below the main diagonal and add the individual generating functions together.  All the seven generic cases, with their corresponding generating functions, are shown in Figure~\ref{fig:cases}.  Adding the individual generating functions together then gives us the overall generating function that was given in Table~\ref{tab:GF}.

\begin{figure}[!htb]
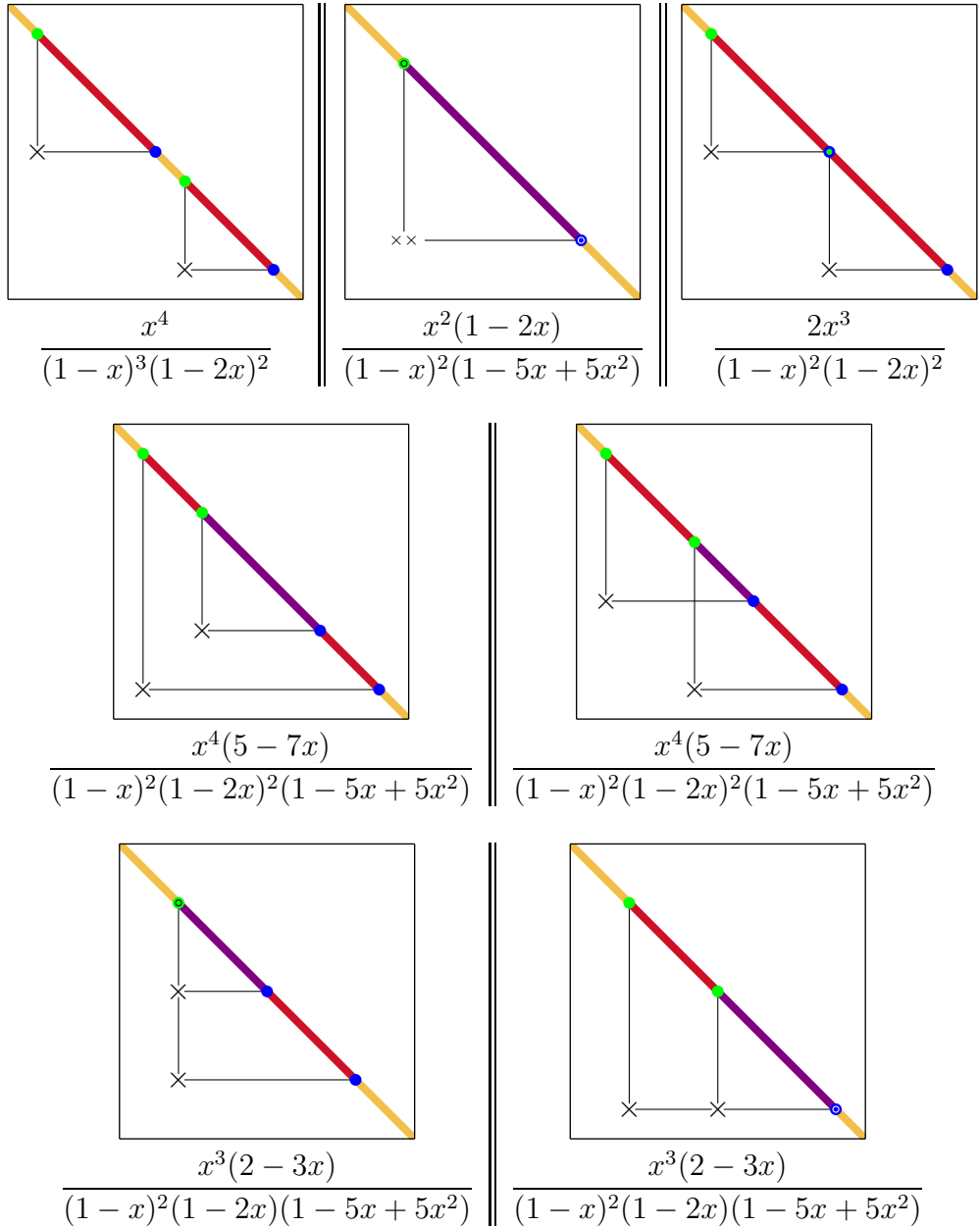

\centering
\begin{tabular}{c}
  \begin{tabular}{c||c||c}
    \picE & \picF & \picG\\
    $\displaystyle\frac{x^4}{(1-x)^3(1-2x)^2}$ & $\displaystyle\frac{x^2(1-2x)}{(1-x)^2(1-5x+5x^2)}$ & $\displaystyle\frac{2x^3}{(1-x)^2(1-2x)^2}$ \\
  \end{tabular}\\
  \begin{tabular}{c}
  \end{tabular}\\
  \begin{tabular}{c||c}
    \picA & \picB\\
    $\displaystyle\frac{x^4(5-7x)}{(1-x)^2(1-2x)^2(1-5x+5x^2)}$ & $\displaystyle\frac{x^4(5-7x)}{(1-x)^2(1-2x)^2(1-5x+5x^2)}$\\
  \end{tabular}\\
  \begin{tabular}{c}
  \end{tabular}\\
  \begin{tabular}{c||c}
    \picC & \picD\\
     $\displaystyle\frac{x^3(2-3x)}{(1-x)^2(1-2x)(1-5x+5x^2)}$ & $\displaystyle\frac{x^3(2-3x)}{(1-x)^2(1-2x)(1-5x+5x^2)}$
  \end{tabular}
\end{tabular}
\caption{All generic $2$-rook placements and corresponding generating functions.}
\label{fig:cases}
\end{figure}

This same process works for determining the generating function of $\big\langle\!{n\atop k}\!\big\rangle_{\!c}$ for any fixed $k$ and $c$.  The main challenge lies in that the number of generic cases that have to be considered grows drastically as we increase $c$ and $k$.  This is demonstrated in Table~\ref{tab:placements}.  It is possible to automate this process, which was used for determining the generating functions for $k=3$ given in Table~\ref{tab:GF}.

\begin{table}[htbf]
\[
\begin{array}{c|rrrrrrr}
&c{=}1&c{=}2&c{=}3&c{=}4&c{=}5&c{=}6&c{=}7\\ \hline
k{=}1&1 \\
k{=}2&4&7 \\
k{=}3&26&68&75 \\
k{=}4&236&940&1090&1105 \\
k{=}5&2752&16645&20360&20790&20821\\
k{=}6&39208&360081&464111&477242&478376&478439 \\
k{=}7&660032&9202170&12492277&12933423&12974826&12977688&12977815 
\end{array}
\]
\caption{The number of generic $c$-rook placements with $k$ rooks below the main diagonal.}
\label{tab:placements}
\end{table}

\section{Conclusion}\label{sec:conc}
The generalized Eulerian numbers are a natural extension of the Eulerian numbers, at least in regards to the interpretation coming from rook placements.  We have also seen that these numbers exhibit a symmetry similar to that of the Eulerian numbers.  It would be interesting to know which other properties and relationships involving Eulerian numbers generalize.  Some natural candidates to try and generalize include the following.
\begin{itemize}
\item Is there a generalization of the recurrence in Proposition~\ref{prop:recurrence} for Eulerian numbers to generalized Eulerian numbers?  Related to this, is there a simple generating function for the generalized Eulerian numbers?
\item Is there a generalization of Worpitzky's identity, $x^n=\sum_k \big\langle\!{n\atop k}\!\big\rangle{x+k\choose n}$, to generalized Eulerian numbers?  Worpitzky's identity is used in counting the number of juggling patterns (see \cite{DropsDescents}), so a generalization might be useful in counting multiplex juggling patterns.
\item Is there a generalization of the identity of Chung, Graham and Knuth \cite{CGR},
\[
\sum_k{\textstyle {a+b\choose k} \big\langle\!{k\atop a-1}\!\big\rangle}=
\sum_k{\textstyle {a+b\choose k} \big\langle\!{k\atop b-1}\!\big\rangle}
\]
(This uses the convention $\big\langle\!{0\atop0}\!\big\rangle=0$.)
\end{itemize}
More information about the Eulerian numbers and various identities and relationships that could be considered are given in Graham, Knuth, Patashnik \cite[Section 6.2]{Concrete}.

We also note the original motivation for investigating these numbers was looking into the mathematics of multiplex juggling.  There is a close connection between the mathematics of juggling and the mathematics of rook placements.  We hope to see this relationship strengthened in future work.

\bigskip

\noindent\textbf{Acknowledgments.}~~The research was conducted at the 2015 REU program held at Iowa State University which was supported by NSF DMS 1457443.


\begin{thebibliography}{99}
\bibitem{rooknotes}
Fred Butler, Mahir Can, Him Haglund, and Jeffrey Remmel, \emph{Rook Theory Notes}, manuscript.  Available online at \url{www.math.ucsd.edu/~remmel/files/Book.pdf}

\bibitem{CGR}
Fan Chung, Ron Graham, and Donald Knuth, \emph{A symmetrical Eulerian identity}, Journal of Combinatorics \textbf{17} (2010), 29--38.

\bibitem{Concrete}
Ron Graham, Donald Knuth, and Oren Patashnik, \emph{Concrete Mathematics: A Foundation for Computer Science (2nd ed.)}, Addison-Wesley Longman Publishing Co., Inc., Boston, 1994.

\bibitem{DropsDescents}
Joe Buhler, David Eisenbud, Ron Graham, and Colin Wright, \emph{Juggling Drops and Descents}, American Math Monthly \textbf{101} (1994) 507--519.

\bibitem{Polster}
Burkhard Polster, \emph{The Mathematics of Juggling}, Springer, New York, 2000.

\bibitem{OEIS}
Neil Sloane, \emph{The On-Line Encyclopedia of Integer Sequences}.  Available online at \url{oeis.org}.
\end{thebibliography}
\end{document}